\documentclass[reqno]{amsart}
\usepackage[scale=0.75, centering, headheight=14pt]{geometry}
\usepackage[latin1]{inputenc}
\usepackage[T1]{fontenc}
\usepackage{lmodern}
\usepackage[english]{babel}

\usepackage{tikz}
\usepackage{bbm}
\usepackage{amsmath,amssymb,amsfonts,amsthm}
\usepackage{mathtools,accents}
\usepackage{mathrsfs}
\usepackage{xfrac}
\usepackage{array} 
\usepackage{aliascnt}
\usepackage{booktabs} % for much better looking tables
\usepackage{array} % for better arrays (eg matrices) in maths

\usepackage{verbatim} % adds environment for commenting out blocks of text & for better verbatim
\usepackage{subfig} % make it possible to include more than one captioned figure/table in a single float

\usepackage{mathrsfs, dsfont}
\usepackage{amssymb}
\usepackage{amsthm}
\usepackage{amsmath,amsfonts,amssymb,esint}
\usepackage{graphics,color}
\usepackage{enumerate}
\usepackage{mathtools,centernot}

\usepackage{microtype}
\usepackage{paralist} % very flexible & customisable lists (eg. enumerate/itemize, etc.)
\usepackage{cases}
\usepackage[initials]{amsrefs}
\allowdisplaybreaks

\usepackage{braket}
\usepackage{bm}

\usepackage[citecolor=blue,colorlinks]{hyperref}
\addto\extrasenglish{}

\usepackage{enumerate}
\usepackage{xcolor}
\usepackage{aliascnt}

\makeatletter
\def\newaliasedtheorem#1[#2]#3{
  \newaliascnt{#1@alt}{#2}
  \newtheorem{#1}[#1@alt]{#3}
  \expandafter\newcommand\csname #1@altname\endcsname{#3}
}
\makeatother

\theoremstyle{plain}
\newtheorem{theorem}{Theorem}[section]

\newaliasedtheorem{lemma}[theorem]{Lemma}
\newaliasedtheorem{prop}[theorem]{Proposition}
\newaliasedtheorem{claim}[theorem]{Claim}
\newaliasedtheorem{corollary}[theorem]{Corollary}
\theoremstyle{remark}

\theoremstyle{definition}
\newaliasedtheorem{definition}[theorem]{Definition}
\newaliasedtheorem{example}[theorem]{Example}

\theoremstyle{remark}
\newaliasedtheorem{remark}[theorem]{Remark}

\numberwithin{equation}{section}

\def\eps{\varepsilon}

\def\R{\mathbb R}
\def\N{{\mathbb N}}% nonnegative integers
\DeclareMathOperator{\spt}{spt}
\DeclareMathOperator{\loc}{loc}
\DeclareMathOperator{\diag}{diag}

\DeclareMathOperator{\id}{id}
\DeclareMathOperator{\intt}{int}
\DeclareMathOperator{\rank}{rank}

\DeclareMathOperator{\Sym}{Sym}

\DeclareMathOperator{\dist}{dist}

\author[G. De Philippis and  R. Tione]{Guido De Philippis \and Riccardo Tione}

\address{G.~De Philippis: \hfill\break Courant Institute of Mathematical Sciences, New York University, 251 Mercer St., New York, NY 10012, USA.}
\email{guido@cims.nyu.edu}

\address{Riccardo Tione
\hfill\break Max Planck Institute for Mathematics in the Sciences, Inselstrasse 22, 04103 Leipzig, Germany}
\email{riccardo.tione@mis.mpg.de}

\thanks{{\bf Acknowledgments.} 
	G.D.P has been partially supported by the NSF grant DMS 2055686 and by the Simons Foundation.}

\title{On Singular strictly convex solutions to the Monge-Amp\`ere equation}
\subjclass[2020]{35J96-35J70-35D30}
\keywords{Degenerate Monge-Amp\`ere equation, singular solutions, convex integration, staircase laminates}
\begin{document}

\maketitle

\begin{abstract}
We show the existence of a strictly convex function $u: B_1 \to \R$ with associated Monge-Amp\`ere measure represented by a function $f$ with $0 < f < 1$ a.e. whose Hessian has a singular part. This extends the work \cite{MON} and answers an open question of \cite[Sec. 6.2(1)]{MONSUR}.
\end{abstract}

\section{Introduction}

This paper is concerned with (low) regularity of convex solutions $u: B_1 \to \R$, $B_1 \subset \R^n$, $n \ge 2$, to the Monge-Amp\`ere equation
\begin{equation}\label{MA}
\begin{cases}
\det D^2 u = f, &\text{ in } B_1,\\
u = 0,& \text{ on } \partial B_1.
\end{cases}
\end{equation}
Here solutions are intended in the  Alexandrov sense, so that 
\[
\det D^2 u = f
\] 
shall be interpreted as 
\[
\mu_u=f dx
\] 
where \(\mu_u\) is the Monge-Amp\`ere measure associated to $u$, see \cite{FIG}. When $0< \delta \le f \le 1$ one can  show that solutions of \eqref{MA} belongs to \(W^{2,1}\) (meaning that the \(\Sym(2)\)-valued measure representing the Hessian of $u$ has no singular part)  and indeed $u \in W^{2,1 + \eps}_{\loc}(B_1)$, for some $\eps(\delta) > 0$,\cite{DPA,DPFS,SCHI}. 

When \(n=2\) and  \(0\le f \le 1\), it is a well known theorem of Alexandrov that  strictly convex  solutions of \eqref{MA} are indeed \(C^1\), see for instance \cite[Theorem 2.17]{FIG}. In view of applications to the semigeostrophic equations it was asked in \cite[Sec. 6.2(1)]{MONSUR} whether solution of \eqref{MA} with \(f\le 1\) can have a singular part of the Hessian measure. In \cite{MON}, Mooney constructed a convex solution to \eqref{MA} whose Hessian has a nontrivial Cantor part, via a sophisticated extension procedure of a convenient one-dimensional convex function. The graph of solutions constructed  however contains open segments, leaving open the question whether an  example can be constructed with  $u$ strictly convex. The main result of this note shows that it is indeed possible.

\begin{theorem}\label{counterex}
There exists a strictly convex $u\in C^1(\overline {B_1})$ with $Du \in BV\cap C^0(\overline {B_1})$ and $u = 0$ on $\partial B_1$ such that
\[
\mu_u = fdx \quad 0 < f \le 1 \text{ a.e.},
\]
but
\[
u \notin W^{2,1}(B_1).
\]
\end{theorem}
With the same method we may also construct $u \in C^{1,\alpha}$, having fixed any $\alpha \in (0,1)$. 

The construction is performed by exploiting  the convex integration method introduced by D. Faraco in \cite{MIL}, based on objects known in the literature as \emph{staircase laminates}. This is an extremely versatile method for creating pathological examples of \emph{concentration}, and has already been used in many contexts, see \cite{CFM,AFSZ,CT, FRT}. In particular, it is used in \cite{FMCO} in connection with the degenerate Monge-Amp\`ere equation
\[
\det(D^2u) = 0 \quad \text{a.e. in }B_1. 
\]
Similar constructions have also been considered in \cite{LM,LMP}. 

\subsection{Notation}

We let $\Sym(2)$ be the space of $2\times 2$ symmetric matrices. 
For any set $E \subset \R^2$, we denote by $\overline E$ its closure, by $\partial E$ its topological boundary and by $\intt E$ its interior. We denote by $|E|$ the 2-dimensional Lebesgue measure of a measurable set $E \subset \R^2$.
\\
\\
For an open set  $\Omega \subset \R^2$ we  say that a function $u: \Omega \to \R$ is piecewise quadratic if it is continuous and there exist countably many pairwise disjoint open sets $\Omega_n \subset \Omega$ with
\[
\left|\Omega\setminus \bigcup_n\Omega_n\right| = 0
\]
and $\left. u\right|_{\Omega_n}$ is a polynomial of second order. For a couple $(x,y) \in \R^2$, we set 
\[
\diag(x,y) \doteq \left(
\begin{array}{cc}
x & 0 \\
0 & y
\end{array}
\right).
\]
We denote by $a\cdot b $ the standard scalar product of $a,b \in \R^2$.
\section{Laminates of finite order}

We wish to use the convex integration methods of \cite{CFM,AFSZ}. We will say that $B,C \in \Sym(2)$ are rank-one connected if
\[
\rank(B-C) = 1.
\] 
The building block of these methods is given by the following:

\begin{lemma}\label{simple}
Let $\Omega \subset \R^2$ be an open and bounded domain. Consider $B,C \in \Sym(2)$ with $\rank(B-C) = 1.$ For any $\lambda \in [0,1]$, let 
\[
A \doteq \lambda B + (1-\lambda)C
\]
Fix any $b \in \R^2$, $c \in \R$ and let $u_0(x)\doteq \frac{1}{2}Ax\cdot x + b\cdot x + c$. Then, for all $\eps >0$, there exists a piecewise quadratic function $u_\eps \in W^{2,\infty}(\Omega)$ with the following properties:
\begin{enumerate}
\item $u_\eps(x) = \frac{1}{2}Ax\cdot x + b\cdot x + c$, $Du_\eps(x) = Ax + b$ on $\partial \Omega$ and $\|u_\eps - u_0\|_{C^1(\overline\Omega)}\le \eps$;
\item $D^2 u_\eps(x) \in B_{\varepsilon}(B)\cup B_{\varepsilon}(C)$ for a.e. $x \in \Omega$;
\item\label{prop:thi} $|\{x \in \Omega: D^2u_\eps(x) = B\}| \ge (1-\eps)\lambda|\Omega|, |\{x \in \Omega: D^2u_\eps(x) =C\}| \ge (1-\eps)(1-\lambda)|\Omega|$;
\item\label{prop:fou} $|\{x \in \Omega: D^2u_\eps(x) \in B_\eps(B)\}| \le (1+\eps)\lambda|\Omega|, |\{x \in \Omega: D^2u_\eps(x)  \in B_\eps(C)\}| \le (1+\eps)(1-\lambda)|\Omega|$;
\end{enumerate}
\end{lemma}
\begin{proof}
A map having the first three properties was constructed in the proof of \cite[Lemma 2.1]{AFSZ}, see in particular \cite[(4.1),(4.2),(4.3),(4.4)]{AFSZ}. Let us denote the family of functions obtained in this way by $v_\eps$. In order to show that there is a map that in addition fulfills \eqref{prop:fou}, we fix $\eps$ and $\lambda$ and assume without loss of generality that $\lambda \le 1 -\lambda$. We can also assume that $\eps >0$ is so small that $B_\eps(B) \cap B_\eps(C) = \emptyset$. Thus, having defined $\eps' \doteq \frac{\eps \lambda}{1-\lambda}$ and the map $v_{\eps'}$, we set $u_\eps \doteq v_{\eps'}$ and claim that this function has the four required properties. Since $\eps' \le \eps$, the first three properties are straightforward from the analogous properties of $v_{\eps'}$. To obtain \eqref{prop:fou}, simply estimate
\begin{align*}
|\{x \in \Omega: D^2u_\eps(x) \in B_\eps(B)\}|& = |\{x \in \Omega: D^2v_{\eps'}(x) \in B_{\eps'}(B)\}| \\
&\le |\Omega| - |\{x \in \Omega: D^2v_{\eps'}(x) = C\}| \overset{\eqref{prop:thi}}{\le} |1 -(1 - \eps')(1-\lambda)||\Omega|.
\end{align*}
Analogously,
\[
|\{x \in \Omega: D^2u_\eps(x) \in B_\eps(C)\}| \le |1-(1-\eps')\lambda||\Omega|.
\]
The fact that $\lambda \le 1-\lambda$ and our choice of $\eps' = \eps \frac{\lambda}{1-\lambda}$ implies that
\[
|1 -(1 - \eps')(1-\lambda)| \le (1 + \eps)\lambda \text{ and } |1-(1-\eps')\lambda|\le (1+\eps)(1-\lambda)
\]
and allows us to conclude the proof.
\end{proof}
This result can be easily iterated, but before stating the next proposition, we start with a definition.
\begin{definition}
Let $\nu,\mu$ be probability atomic measures on $\Sym(2)$. Let $\nu = \sum_{i= 1}^N\lambda_i\delta_{A_i}$. We say that $\mu$ can be obtained via \emph{elementary splitting from }$\nu$ if for some $i \in \{1,\dots,N\}$, there exist matrices $B,C \in \Sym(2)$, $\lambda \in [0,1]$ such that
\[
\rank(B-C) = 1, \quad A_i = sB + (1-s)C,
\]
for some $s \in (0,1)$ and
\[
\mu = \nu +\lambda\lambda_i(-\delta_{A_i} + s\delta_B + (1-s)\delta_C).
\]
A probability measure $\nu = \sum_{i= 1}^r\lambda_i\delta_{A_i}$ is called a \emph{laminate of finite order} if there exists a finite number of measures $\nu_1,\dots,\nu_r \in$ such that
\[
\nu_1 = \delta_X,\quad \nu_r = \nu
\]
and $\nu_{j + 1}$ can be obtained via elementary splitting from $\nu_j$, for every $j\in \{1,\dots,N-1\}$.
\end{definition}
A simple inductive procedure starting from Lemma \ref{simple} yields then the next result, see also \cite[Proposition 2.3]{AFSZ} or \cite[Proposition 4.6]{ST}.
\begin{lemma}\label{ind}
Let $\Omega \subset \R^2$ be an open domain. Consider a laminate of finite order $\nu = \sum_{i = 1}^r\lambda_i\delta_{A_i}$ with barycenter $A \in \Sym(2)$, i.e.: 
\[
A = \int_{\R^{2\times 2}} Xd\nu(X).
\]
Fix any $b \in \R^2$, $c \in \R$ and let $u_0(x)\doteq \frac{1}{2}Ax\cdot x + b\cdot x + c$. Then, for all $\eps >0$, there exists a piecewise quadratic function $u_\eps \in W^{2,\infty}(\Omega)$ with the following properties:
\begin{enumerate}
\item $u_\eps(x) = \frac{1}{2}Ax\cdot x + b\cdot x + c$, $Du_\eps(x) = Ax + b$ on $\partial \Omega$ and $\|u_\eps - u_0\|_{C^1(\overline\Omega)}\le \eps$;
\item $D^2 u_\eps(x) \in  \bigcup_{i = 1}^rB_{\varepsilon}(A_i)$ for a.e. $x\in \Omega$;
\item $|\{x \in \Omega: D^2u_\eps(x) = A_i\}| \ge (1-\eps)\lambda_i|\Omega|, \forall i$;
\item $|\{x \in \Omega: D^2u_\eps(x) \in B_\eps(A_i)\}| \le (1+\eps)\lambda_i|\Omega|, \forall i$.
\end{enumerate}
\end{lemma}

Operatively, what one usually does is the following. Starting with a matrix $A \in \Sym(2)$, one \emph{splits} $A$ into $B$ and $C$ with weight $\lambda$ if
\[
\rank(B-C) = 1 \text{ and } A = \lambda B + (1-\lambda)C. 
\]
We will use the short-hand notation:
\begin{equation}\label{splitmap}
\delta_A \mapsto \nu_1 = \lambda \delta_B + (1-\lambda)\delta_C.
\end{equation}
By definition, $\nu_1$ is a laminate of finite order. One can then reiterate this reasoning by splitting $C$ into $D$ and $E$ with weight $\mu$, provided
\[
\rank(D - E) = 1 \text{ and } C = \mu D + (1-\mu) E,
\] 
so that
\[
\delta_C \mapsto \mu\delta_D + (1-\mu)\delta_E.
\]
By construction, the measure
\[
\nu_2 \doteq \lambda \delta_B + (1-\lambda)\mu\delta_D + (1-\lambda)(1-\mu)\delta_E.
\]
is again a laminate of finite order. Before concluding this section we show one last result, which will be useful in the proof of Theorem \ref{counterex}.

\begin{lemma}\label{tecbound}
Let $\Omega \subset \R^2$ be open and let $u: \Omega \to \R$ be a $W^{2,1}(\Omega)$ piecewise quadratic function. Let $\{\Omega_n\}$ be the collection of pairwise disjoint open sets such that
\[
\left|\Omega\setminus \bigcup_n\Omega_n\right| = 0
\]
and $\left. u\right|_{\Omega_n}$ is a polynomial of second order. Define, for any Borel $F \subset \Sym(2)$, the set
\[
E \doteq \{x \in \Omega: \text{x is a Lebesgue point for }D^2u(x) \text{ and } D^2u(x) \in F\},
\]
Then,
\begin{equation}\label{cons1}
\left|E \setminus \bigcup_{n: \Omega_n\cap E \neq \emptyset}\Omega_n\right| = 0
\end{equation}
and
\begin{equation}\label{cons2}
|\partial E \cap \Omega| = 0.
\end{equation}
\end{lemma}
\begin{proof}
Notice that, since $u|_{\Omega_n}$ is a polynomial of second order, $E$ enjoys the property
\begin{equation}\label{openness}
E\cap \Omega_n \neq \emptyset \Rightarrow  \Omega_n \subset E, \quad \forall n\in \N.
\end{equation}
Thus, either $|E| = 0$, in which case \eqref{cons1} is obvious, or $|E| > 0$. In this case, \eqref{openness} shows
\[
\bigcup_{n: \Omega_n\cap E \neq \emptyset}\Omega_n \subset \intt E \subset E
\]
and \eqref{cons1} is immediate. To show \eqref{cons2}, suppose by contradiction that for some $F$,
\[
|\partial E \cap \Omega| > 0. 
\]
In this case we must have that, for some $m \in \N$,
\[
\Omega_m \cap \partial E \cap \Omega \neq \emptyset.
\]
Let $x \in \Omega_m\cap \partial E \cap \Omega$. On one hand, we must have that $x\notin E$, since otherwise \eqref{openness} yields $\Omega_m \subset E$ and hence $x \notin \partial E$. On the other hand, since $\Omega_m$ intersects $\partial E$, then $\Omega_m$ must intersect $E$ and hence $x \in E$ by \eqref{openness}, which yields the required contradiction.
\end{proof}

\section{Proof of Theorem \ref{counterex}}
We divide the proof into steps.
\\
\\
\fbox{Construction of the laminates.}
We will work in the space of diagonal matrices, where we use coordinates
\[
(x,y) \in \R^2 \mapsto
\left(
\begin{array}{cc}
x & 0 \\
0 & y
\end{array}
\right).
\]
Notice that in the space of diagonal matrices two matrices $(x,y)$ and $(a,b)$ are rank-one connected if and only if $x = a$ or $y = b$. We define the laminate of finite order $\nu_n$ in the following way. First, we let for all $n \ge 1$:
\begin{equation}\label{explicit}
x_n = 4^{-n}, \quad y_n = 2^n, \quad b_n = \frac{3}{4}2^{-n}, \quad z_n = 4^{-n}.
\end{equation}
Moreover, set
\[
x_0 \doteq \frac{1}{2}, \quad y_0 \doteq \frac{1}{2}, \quad b_0 \doteq \frac{3}{4}, \quad z_0 = \frac{1}{8}.
\]
Observe that $0< x_{n + 1} < x_n < b_n$ for all $n \in \N$. Thus, we can first split horizontally
\[
\delta_{(x_n,y_n)} \mapsto \frac{b_n - x_n}{b_n - x_{n + 1}}\delta_{(x_{n+1}, y_n)} + \frac{x_n - x_{n + 1}}{b_n - x_{n + 1}}\delta_{(b_n,y_n)}.
\]
 Next, since $0 < z_n < y_n < y_{n + 1}$ for all $n\in \N$, we are allowed to split vertically
\[
\delta_{(x_{n+1}, y_n)} \mapsto \frac{y_n - z_n}{y_{n + 1} - z_n}\delta_{(x_{n+1}, y_{n + 1})} + \frac{y_{n + 1} - y_n}{y_{n + 1} - z_n}\delta_{(x_{n+1}, z_n)}.
\]
By construction, the probability measure
\[
\nu_n \doteq \frac{b_n - x_n}{b_n - x_{n + 1}}\frac{y_n - z_n}{y_{n + 1} - z_n}\delta_{(x_{n+1}, y_{n + 1})} + \frac{b_n - x_n}{b_n - x_{n + 1}}\frac{y_{n + 1} - y_n}{y_{n + 1} - z_n}\delta_{(x_{n+1}, z_n)} + \frac{x_n - x_{n + 1}}{b_n - x_{n + 1}}\delta_{(b_n,y_n)}
\]
is a laminate of finite order. 
\begin{comment}Let us notice the following property of our sequences, which will be crucial in the following:
\begin{equation}\label{det<1}
0 < x_{n}y_{n} < 1,\quad 0 < x_{n + 1}z_n < 1, \quad 0 < b_ny_n < 1, \quad \forall n \in \N,
\end{equation}
so that for all $n \in \N$:
\[
\spt(\nu_n) \subset \{X \in \Sym(2): X >0, 0 < \det(X) < 1 \},
\]
where $X > 0$ has to be intended in the sense of quadratic forms.
\end{comment}
 Let for all $n \in \N$
\[
A_n = \diag(x_n,y_n),\quad B_n = \diag(x_{n + 1},z_n), \quad C_n = \diag(b_n,y_n)
\]
and
\begin{equation}\label{alphai}
\alpha_n \doteq \frac{b_n - x_n}{b_n - x_{n + 1}}\frac{y_n - z_n}{y_{n + 1} - z_n}, \quad \beta_n\doteq  \frac{b_n - x_n}{b_n - x_{n + 1}}\frac{y_{n + 1} - y_n}{y_{n + 1} - z_n}, \quad \gamma_n \doteq \frac{x_n - x_{n + 1}}{b_n - x_{n + 1}}.
\end{equation}
\fbox{The base step.}
The reasoning to construct the function $u$ of the statement of the present Theorem is inductive. We start with 
\[
u_0(z) \doteq \frac{1}{2}A_0z\cdot z, \quad \forall z \in \R^2.
\]
Next we define:
\begin{equation}\label{epsn}
\eps_n \doteq 10^{-n}.
\end{equation}
We can consider $u_1: B_1 \to \R$, which is obtained employing Lemma \ref{ind} with $\nu_0$ and $\eps_1$ in place of $\eps$. In other words, $u_1 \in W^{2,\infty}$ is a piecewise quadratic function which satisfies:
\begin{itemize}
\item $u_1 = u_0$, $Du_1 = Du_0$ on $\partial B_1$ and $\|u_1 - u_0\|_{C^1(\overline B_1)} \le \eps_1$;
\item $D^2u_1 \in B_{\eps_1}(A_1)\cup B_{\eps_1}(B_0) \cup B_{\eps_1}(C_0)$ a.e. in $B_1$;
\item $|\{z \in B_1: D^2u_1(z) = A_1\}| \ge \alpha_0(1-\eps_1)|B_1|.$
\item $|\{z \in B_1: D^2u_1(z) \in B_{\eps_1}(A_1)\}| \le (1+\eps_1)\alpha_0|B_1|.$
\end{itemize}
Consider the set $E_1 = \{z \in \Omega_0 \doteq B_1: D^2u_1(z) = A_1\}$. Since $|E_1| >0$, Lemma \ref{tecbound} allows us to choose another open set $\Omega_1 \subset E_1$ which is a finite union of balls whose closures are pairwise disjoint and
\[
|E_1\setminus \Omega_1| \le \eps_1|E_1|.
\]
\begin{comment}
Moreover, we can assume that
\begin{equation}\label{omega1}
\dist(\Omega_1,\partial B_1) > 0.
\end{equation}
\end{comment}
Since $u_1$ is piecewise quadratic, we can subdivide $\Omega_1$ into countably many smaller open sets $\{F_n\}_{n \in \N}$ on which $u_1$ is quadratic. Let $u_1(z) = \frac{1}{2}A_1z\cdot z + b\cdot z + c$ on $F_n$ for some $b \in \R^2, c\in \R$. Then, we can replace $u_1$ on $F_n$ using a function $w_{2,n}$ given again by Lemma \ref{ind} with $\nu_1$ and $\eps_2$ in place of $\nu$ and $\eps$. We then define
\[
u_2 \doteq
\begin{cases}
u_1 \text{ on }B_1 \setminus \bigcup_{n}\overline {F_n}\\
w_{2,n} \text{ on }\overline {F_n}.
\end{cases}
\]
By Proposition \ref{ind}
\[
w_{2,n} =  \frac{1}{2}A_1z\cdot z + b\cdot z + c \text{ and } Dw_{2,n} =  A_1z + b \text{ on }\partial F_n,
\]
and hence $u_2$ is still in $W^{2,\infty}(B_1)$. We will now set up the inductive procedure.
\\
\\ 
\fbox{The inductive step.}
Suppose we are given an open, non-empty set $\Omega_{j -1} \subset B_1$ and a piecewise quadratic function $u_j \in W^{2,\infty}(B_1)$ satisfying
\[
D^2u_{j} \in \bigcup_{k = 1}^{j} B_{\eps_{k}}(A_{k})\cup  \bigcup_{k = 0}^{j-1}B_{\eps_{k +1}}(B_k)\cup B_{\eps_{k+1}}(C_k) \text{ a.e. in $B_1$}.
\]
Let
\begin{equation}\label{Ejdef}
E_j = \{z \in \Omega_{j-1}: D^2u_j(z) = A_{j}\},
\end{equation}
which we assume to have positive measure. By Lemma \ref{tecbound}, we can choose $\Omega_j = \bigcup_{i = 1}^{N_j}B_{r_i}(x_i)$ with $r_i \le \frac{1}{j}$ for all $i$, $\overline{B_{r_i}(x_i)} \subset E_j$, and
\[
\overline{B_{r_i}(x_i)} \cap \overline{B_{r_k}(x_k)} = \emptyset, \quad \forall i \neq k,
\]
so that
\begin{equation}\label{segment}
\text{$\overline{\Omega_j}$ does not contain segments of length larger than $\frac{2}{j}$}.
\end{equation}
The balls forming $\Omega_j$ are chosen in such a way that
\begin{equation}\label{close}
|E_j\setminus \Omega_j| \le \eps_j|E_j|.
\end{equation}
Finally notice that
\begin{equation}\label{bdry}
|\partial \Omega_j| = 0.
\end{equation}
We can split $\Omega_j = \bigcup_{n}\Omega_{j,n}$ such that on each open set $\Omega_{j,n}$, $u_j$ is quadratic. In particular, we have
\[
u_j(z) = \frac{1}{2}A_{j}z \cdot z + b_{j,n}\cdot z + c_{j,n} \text{ on }\Omega_{j,n}, \text{ where } A_{j} = \diag(x_{j},y_{j}), b_{j,n} \in \R^2, c_{j,n} \in \R.
\]
On each $\Omega_{j,n}$, we can consider as boundary datum $u_j$ and use the laminate of finite order $\nu_{j}$ in conjunction with Lemma \ref{ind} to construct a map $w_{j,n} \in W^{2,\infty}$ with the following properties:
\begin{enumerate}[(a)]
\item $w_{j,n} = u_j$, $Dw_{j,n} = Du_j$ on $\partial \Omega_{j,n}$ and $\|w_{j,n} - u_j\|_{C^1(\overline{\Omega_{j,n}})} \le \eps_{j +1}$;
\item $D^2w_{j,n} \in B_{\eps_{j + 1}}(A_{j+1})\cup B_{\eps_{j + 1}}(B_{j})\cup B_{\eps_{j + 1}}(C_j)$ a.e. in $\Omega_{j,n}$;
\item furthermore
\[
|\{x \in \Omega_{j,n}: D^2w_{j,n}(x) = A_{j+ 1}\}| \ge (1-\eps_{j + 1})\alpha_{j} |\Omega_{j,n}|
\]
and
\[
|\{x \in \Omega_{j,n}: D^2w_{j,n}(x) \in B_{\eps_{j  + 1}}(A_{j+ 1})\}| \le (1+\eps_{j +1})\alpha_{j} |\Omega_{j,n}|.
\]
\end{enumerate}
Thus, if we define
\begin{equation}\label{uj+1}
u_{j + 1}(x) \doteq \begin{cases}
u_j(x), \text{ if } x \in B_1 \setminus \overline{\Omega_j}\\ 
w_{j,n}(x), \text{ if } x \in \overline{\Omega_{j,n}},
\end{cases}
\end{equation}
 we see that $u_{j + 1} \in W^{2,\infty}(B_1)$ enjoys the following properties:
\begin{enumerate}[(i)]
\item\label{uj} $u_{j + 1} = u_j$, $Du_{j + 1} = Du_j$ on $\partial B_1$ and $\|u_{j + 1} - u_j\|_{C^1(\overline{B_1})} \le \eps_{j +1}$;
\item\label{convv} $D^2u_{j + 1} \in \bigcup_{k = 1}^{j + 1} B_{\eps_{k}}(A_{k})\cup \bigcup_{k = 0}^j B_{\eps_{k+1}}(B_k)\cup B_{\eps_{k+1}}(C_k)$ a.e. in $B_1$;
\item\label{measures} furthermore 
\[
|\{x \in \Omega_j: D^2u_{j+1}(x) = A_{j+ 1}\}| \ge (1-\eps_{j + 1})\alpha_{j} |\Omega_{j}|
\]
and
\[
|\{x \in \Omega_j: D^2u_{j+1}(x) \in B_{\eps_{j  +1}}(A_{j+ 1})\}| \le (1+\eps_{j +1})\alpha_{j} |\Omega_{j}|.
\]
\end{enumerate}
\fbox{Definition of the required map and first properties.}
The inductive scheme gives us a sequence $\{u_j\}_{j \in \N}$. Extend each function $u_j$ to $B_3$ by taking
\[
\bar u_j \doteq \begin{cases}
u_j, \text{ in }\overline{B_1}\\
u_0, \text{ in }B_3\setminus \overline{B_1}.
\end{cases}
\]
First, by \eqref{uj} we have that
\[
\sum_{j=1}^\infty\|\overline u_{j + 1} - \overline u_j\|_{C^1(\overline{B_2})} < + \infty.
\]
Thus we can define $\overline u_\infty = \lim_{j}\overline u_j$, where the limit is taken in the $C^1(\overline B_2)$ topology. As for $u_j$, we will denote $u_\infty$ the restriction of $\overline u_\infty$ to $\overline B_1$. Moreover, 
\[
u_\infty = u_0 = \frac{1}{4}(x^2 + y^2) = \frac{1}{4} \text{ on $\partial B_1$}
\]
again by \eqref{uj}. We claim that $u \doteq u_\infty - \frac{1}{4}$ gives the required counterexample. We will always work with $u_\infty$ in the following, since anyway the subtraction of the constant $\frac{1}{4}$ does not change the properties we are interested in. Due to our choice of $\{\eps_j\}_{j \in \N}$, $\{x_j\},\{y_j\},\{b_j\}$ and $\{z_j\}$, see \eqref{explicit}-\eqref{epsn}, and exploiting \eqref{convv}, we also find $\rho_j >0$ such that
\begin{equation}\label{ro}
D^2\overline u_j(x) \ge \rho_j\id
\end{equation}
a.e. in $B_3$ in the sense of quadratic forms. In particular, this yields that $\overline u_j$ is convex for all $j$ and thus $\overline u_\infty$ is convex too. Moreover in the sense of measures
\begin{equation}\label{wconv}
D^2\overline u_j \overset{*}{\rightharpoonup} D^2\overline u_\infty, \text{ in } B_2.
\end{equation}
Since $u_j \in W^{2,\infty}(B_1)$ for all $j$ we have
\[
\mu_{u_j} = \det(D^2u_j)dz
\]
and once again our choice of $\{\eps_j\}_j, \{x_j\}, \{y_j\}, \{b_j\}$ and $\{z_j\}$ in conjunction with \eqref{convv} yield
\[
  0 <  \det(D^2u_j)(x) < 1 \text{ a.e. in }B_1.
\]
By \cite[Proposition 2.6]{FIG}, it follows that $\mu_{u_\infty}$ is the weak-$*$ limit of $\mu_{u_j}$ and it is therefore represented as
\[
\mu_{u_\infty} = fdz, \text{ and }0 \le f \le 1.
\]
Up to now, we have showed that $u_\infty$ enjoys the following properties:
\begin{itemize}
\item $u_\infty$ is convex;
\item $u_\infty \in C^1(\overline {B_1})$ and $Du \in BV\cap C^0(\overline {B_1})$;
\item $u_\infty = \frac{1}{4}$ on $\partial B_1$, so that $u = 0$ on $\partial B_1$;
\item $\mu_{u_\infty} = fdz$ and $0 \le f \le 1$ a.e..
\end{itemize}
To conclude, we still need to show that
\begin{itemize}
\item $0 < f$ a.e.;
\item $u_\infty$ is strictly convex;
\item $D^2u_\infty$ is singular with respect to the Lebesgue measure.
\end{itemize}
In order to so, we define the closed set
\[
\Omega_\infty \doteq \bigcap_{j=1}^\infty \overline \Omega_j.
\]
We claim that
\begin{equation}\label{claim1}\tag{Claim 1}
|\Omega_j| \overset{\eqref{bdry}}{=} |\overline \Omega_j| \to 0 \text{ as } j \to \infty
\end{equation}
and that there exists $c > 0$ such that
\begin{equation}\label{claim2}\tag{Claim 2}
c \le \int_{\Omega_j}\partial_{22}u_{j}  \text{ for all } j.
\end{equation}
First assume these claims and let us show how they allow us to conclude the proof.
\\
\\
\fbox{$f > 0$ a.e..} \eqref{claim1} implies that $|\Omega_\infty| = 0$. Thus it suffices to prove that $f > 0$ a.e. on $B_1\setminus \Omega_\infty = \bigcup_{j}(B_1\setminus\overline{\Omega_j})$. On the open set $B_1\setminus \overline{\Omega_{j}}$ we have, by \eqref{uj+1} and the fact that $\Omega_{j + 1} \subset \Omega_j$ for all $j \in \N$,
\[
u_\infty = u_j,
\]
and hence, since $u_j \in W^{2,\infty}$,
\[
\mu_{u_\infty} = fdz = \mu_{u_j} = \det(D^2u_j)dz \overset{\eqref{ro}}{>} 0, \text{ a.e. on }  B_1\setminus\overline{\Omega_{j}}.
\]
\\
\\
\fbox{$u_\infty$ is stricly convex.} Suppose by contradiction this does not hold. Then there exists a segment $\sigma \subset B_1$ such that $u_\infty$ is affine on $\sigma$. First, we note that
\[
\sigma \not\subset \Omega_\infty.
\]
Indeed, if by contradiction $\sigma \subset \Omega_\infty$, then by definition $\sigma \subset \overline \Omega_j$ for all $j \in \N$. This is in contradiction with our construction of $\Omega_j$, see \eqref{segment}, for $j$ large enough. Since $\Omega_\infty$ is closed, we find another (nontrivial) segment
\[
\sigma' \subset \sigma
\]
 with
\[
\sigma' \cap \Omega_\infty = \emptyset.
\]
Thus
\[
\sigma' = \sigma' \cap \bigcup_{j}(B_1\setminus\overline \Omega_j).
\]
Let $x_0 \in \sigma'$ and $j_0 \in \N$ such that $x_0 \in B_1\setminus\overline \Omega_{j_0}$. By \eqref{uj+1}, we find a small ball $B_\eps(x_0) \subset B_1$ such that $u_\infty = u_{j_0}$ on $B_\eps(x_0)$. Thus,
\[
u_\infty = u_{j_0} \text{ on } B_\eps(x_0) \cap \sigma' \neq \emptyset.
\]
By \eqref{ro}, $u_{j_0}$ is strictly convex on $B_\eps(x_0)$, which is in contradiction with the fact that $u_\infty$ is affine on $\sigma$. It follows that $u_\infty$ is strictly convex in $B_1$.
\\
\\
\fbox{$D^2u_\infty$ is singular.} It suffices to prove that $\partial_{22}u_\infty$ is singular. Notice that by convexity, $\partial_{22}u_j,\partial_{22}u_\infty \ge 0$ in the sense of measures. Now, by construction, $\Omega_{j + 1} \subset \Omega_j$ for all $j \in \N$. Fix $j$ and take any $\ell > j$. Then:
\[
c \overset{\eqref{claim2}}{\le} \int_{\Omega_{\ell}}\partial_{22} u_{\ell} dz\le \int_{\Omega_j}\partial_{22} u_{\ell} dz \overset{\eqref{bdry}}{=} \int_{\overline\Omega_j}\partial_{22} u_{\ell} dz.
\]
By \eqref{wconv}, $\partial_{22} u_\ell$ weakly-$*$ converges in the sense of measures to $\partial_{22}u_\infty$ and hence by \cite[Theorem 1.40(ii)]{EVG}, we may use the previous chain of inequalities to conclude that, for all $j \in \N$,
\[
c \le \limsup_{\ell \to \infty}\int_{\overline\Omega_j}\partial_{22} u_{\ell} dz \le \partial_{22}u_\infty(\overline \Omega_j).
\]
Exploiting once more \eqref{claim1}, we conclude that $\partial_{22}u_\infty$ and hence $D^2u_\infty$ is singular with respect to the Lebesgue measure.
\\
\\
We now turn to the proof of the claims.
\\
\\
\fbox{Proof of \eqref{claim1}.}
We wish to show that there exist constants $c_1, c_2 >0$ such that
\begin{equation}\label{almostmeas}
c_1 \le \liminf_{j \to \infty}|\Omega_j|y_j\le \limsup_{j \to \infty}|\Omega_j|y_j \le c_2.
\end{equation}
This would imply \eqref{claim1}. We start by noticing that
\begin{equation}\label{aconv}
k_0 =  \prod_{i =0}^\infty (1-\eps_{i + 1})^{2} \in (0,\infty), \quad k_0' = \prod_{i =1}^\infty (1+\eps_{i}) \in (0,\infty),
\end{equation}
\begin{equation}\label{k1k2}
 k_1 \doteq \prod_{i = 1}^\infty  \frac{b_i - x_i}{b_i - x_{i + 1}} = \prod_{i = 1}^\infty  \frac{1- \frac{x_i}{b_i}}{1 - \frac{x_{i + 1}}{b_i}}  \in (0,\infty),\quad k_2 \doteq  \prod_{i = 1}^\infty\frac{1 - \frac{z_i}{y_i}}{1- \frac{z_i}{y_{i + 1}}} \in (0,\infty).
\end{equation}
To prove these  statements recall the following elementary fact: if $\{a_i\}_i$ is a sequence of real numbers, with $1 + a_i > 0$ for all $i \in \N$, then 
%and
%\[
%a_i \to 0 \text{ as }i\to \infty, 
%\] 
%then
%\begin{equation}\label{forj}
%\prod_{i = 1}^j(1 + a_i) = e^{\sum_{i = 1}^j\ln(1 + a_i)}.
%\end{equation}
in order for
\begin{equation}\label{res}
\lim_{j\to\infty}\prod_{i = 1}^j(1 + a_i)\in (0,+\infty)
\end{equation}
it is sufficient that
\begin{equation}\label{suff}
\sum_{i=1}^\infty|a_i| < + \infty.
\end{equation}
%Indeed, if \eqref{suff} holds, then, for sufficiently large $I$
%\begin{align*}
%\sum_{i=I}^\infty|\ln(1 + a_i) &\le 2\sum_{i=I}^\infty|a_i|
%\end{align*}
%whence \eqref{res} follows, combining the latter with \eqref{forj}. 
By \eqref{explicit}-\eqref{epsn} we immediately obtain \eqref{suff} in the cases considered in \eqref{aconv}-\eqref{k1k2}, so these hold. Having shown \eqref{aconv}-\eqref{k1k2}, we can move to \eqref{almostmeas}. From the first line of \eqref{measures} and by \eqref{Ejdef}, we find
\[
(1-\eps_{j + 1})\alpha_j|\Omega_j| \le |E_{j + 1}|.
\]
Using \eqref{close}, we further get
\begin{equation}\label{indomega}
(1-\eps_{j + 1})\alpha_{j}|\Omega_j| \le (1 - \eps_{j + 1})^{-1}|\Omega_{j + 1}|, \quad \forall j \in \N.
\end{equation}
Inductively we can use \eqref{indomega} to write for all $j \in \N$
\[
k_0\left(\prod_{i = 1}^{j-1}\alpha_i\right)|\Omega_1| = \left(\prod_{i =0}^\infty (1-\eps_{i + 1})^2\right) \left(\prod_{i = 1}^{j-1}\alpha_i\right)|\Omega_1| \le |\Omega_j|.
\]
Moreover, using the fact that $\Omega_j \subset E_j$ for all $j \in \N$ we find:
\begin{equation}\label{alfaj}
|\Omega_{j + 1}| \le |E_{j + 1}| \overset{\eqref{Ejdef}}{\le} |\{x \in \Omega_{j}: D^2u_{j + 1}(x) \in B_{\eps_{j + 1}}(A_{j + 1})\}| \overset{\eqref{measures}}{\le} \alpha_{j}(1+\eps_{j +1})|\Omega_j|, \quad \forall j \in \N.
\end{equation}
Using \eqref{alfaj} inductively, we obtain
\[
|\Omega_j| \le |\Omega_1|\prod_{i = 1}^{j-1}\alpha_i\prod_{i = 2}^{j}(1+\eps_{i}) \le |\Omega_1|\prod_{i = 1}^{j-1}\alpha_i\prod_{i = 2}^{\infty}(1+\eps_{i})  \le k_0'|\Omega_1|\prod_{i = 1}^{j-1}\alpha_i, \quad \forall j \in \N.
\]
As $|\Omega_1| > 0$, in order to conclude \eqref{almostmeas} we only need to study the asymptotic behaviour of $\prod_{i = 1}^j\alpha_i$. We have:
\[
\lim_{j \to \infty}y_{j + 1}\prod_{i = 1}^{j}\alpha_i \overset{\eqref{alphai}}{=} \lim_{j \to \infty}y_{j + 1}\prod_{i = 1}^j\frac{b_i - x_i}{b_i - x_{i + 1}}\frac{y_i - z_i}{y_{i + 1} - z_i} \overset{\eqref{k1k2}}{=} y_1k_1k_2 = 2k_1k_2,
\]
where we used that $y_1 = 2$. This shows \eqref{almostmeas} and concludes the proof of this claim.
\\
\\
\fbox{Proof of \eqref{claim2}.}
By \eqref{Ejdef} and since $\Omega_j \subset E_j$, we have
\[
\int_{\Omega_j}\partial_{22} u_{j}dz = y_{j}|\Omega_j|
\]
and hence \eqref{claim2} follows from \eqref{almostmeas}.

\bibliographystyle{plain}
\bibliography{Area}

\begin{bibdiv}
\begin{biblist}

\bib{AFSZ}{article}{
      author={Astala, Kari},
      author={Faraco, Daniel},
      author={Sz{\'{e}}kelyhidi, L{\'{a}}szl{\'{o}}},
       title={{C}onvex integration and the {$L^p$} theory of elliptic
  equations},
        date={2008-01},
     journal={Annali della Scuola Normale - Classe di Scienze},
      volume={7},
       pages={1\ndash 50},
}

\bib{CT}{article}{
      author={Colombo, Maria},
      author={Tione, Riccardo},
       title={{Non-classical solutions of the $p$-Laplace equation}},
        date={2022-01},
     journal={To appear in: JEMS},
      eprint={2201.07484},
}

\bib{CFM}{article}{
      author={Conti, Sergio},
      author={Faraco, Daniel},
      author={Maggi, Francesco},
       title={{A New Approach to Counterexamples to $L^1$ Estimates: Korn's
  Inequality, Geometric Rigidity, and Regularity for Gradients of Separately
  Convex Functions}},
        date={2004-12},
     journal={Archive for Rational Mechanics and Analysis},
      volume={175},
      number={2},
       pages={287\ndash 300},
}

\bib{DPA}{article}{
      author={De~Philippis, Guido},
      author={Figalli, Alessio},
       title={{$W^{2,1}$ regularity for solutions of the
  Monge{\textendash}Amp{\`{e}}re equation}},
        date={2012-04},
     journal={Inventiones mathematicae},
      volume={192},
      number={1},
       pages={55\ndash 69},
}

\bib{DPFS}{article}{
      author={De~Philippis, Guido},
      author={Figalli, Alessio},
      author={Savin, Ovidiu},
       title={{A note on interior $W^{2,1 + \varepsilon}$ estimates for the
  Monge-Amp{\`{e}}re equation}},
        date={2012-12},
     journal={Mathematische Annalen},
      volume={357},
      number={1},
       pages={11\ndash 22},
}

\bib{EVG}{book}{
      author={Evans, Lawrence~C.},
      author={Gariepy, Ronald~F.},
       title={Measure theory and fine properties of functions},
   publisher={Chapman , Hall/CRC},
        date={2015},
}

\bib{MIL}{article}{
      author={Faraco, Daniel},
       title={Milton's conjecture on the regularity of~solutions to isotropic
  equations},
        date={2003-09},
     journal={Annales de l{\textquotesingle}Institut Henri Poincar{\'{e}} C,
  Analyse non lin{\'{e}}aire},
      volume={20},
      number={5},
       pages={889\ndash 909},
}

\bib{FMCO}{article}{
      author={Faraco, Daniel},
      author={Mora-Corral, Carlos},
      author={Oliva, Marcos},
       title={Sobolev homeomorphisms with gradients of low rank via laminates},
        date={2016-03},
     journal={Advances in Calculus of Variations},
      volume={11},
}

\bib{FRT}{article}{
      author={Fern{\'{a}}ndez-Real, Xavier},
      author={Tione, Riccardo},
       title={Improved regularity of second derivatives for subharmonic
  functions},
        date={2023-09},
     journal={Proceedings of the American Mathematical Society},
}

\bib{FIG}{book}{
      author={Figalli, Alessio},
       title={{The Monge-Amp\`ere Equation and Its Applications}},
      series={Zurich Lectures in Advanced Mathematics},
   publisher={European Mathematical Society},
         how={{Z}urich {L}ectures in {A}dvanced {M}athematics},
        date={2017},
}

\bib{LM}{article}{
      author={Liu, Zhuomin},
      author={Mal{\'{y}}, Jan},
       title={{A strictly convex Sobolev function with null Hessian minors}},
        date={2016-05},
     journal={Calculus of Variations and Partial Differential Equations},
      volume={55},
      number={3},
}

\bib{LMP}{article}{
      author={Liu, Zhuomin},
      author={Mal{\'{y}}, Jan},
      author={Pakzad, Mohammad~Reza},
       title={{Approximation by mappings with singular Hessian minors}},
        date={2018-11},
     journal={Nonlinear Analysis},
      volume={176},
       pages={209\ndash 225},
}

\bib{MON}{article}{
      author={Mooney, Connor},
       title={{Some counterexamples to Sobolev regularity for degenerate
  Monge-Amp\`ere equations}},
        date={2016-07},
     journal={Analysis {\&} {PDE}},
      volume={9},
      number={4},
       pages={881\ndash 891},
}

\bib{MONSUR}{article}{
      author={Mooney, Connor},
       title={{The Monge-Amp\`ere Equation}},
        date={2018},
     journal={Rend. Semin. Mat. Univ. Politec. Torino},
      volume={76},
       pages={93\ndash 113},
}

\bib{SCHI}{article}{
      author={Schmidt, Thomas},
       title={{$W^{2,1 + \varepsilon}$ estimates for the Monge-Amp{\`{e}}re
  equation}},
        date={2013-06},
     journal={Advances in Mathematics},
      volume={240},
       pages={672\ndash 689},
}

\bib{ST}{article}{
      author={Sorella, Massimo},
      author={Tione, Riccardo},
       title={The four-state problem and convex integration for linear
  differential operators},
        date={2023-02},
     journal={Journal of Functional Analysis},
      volume={284},
      number={4},
       pages={109785},
}

\end{biblist}
\end{bibdiv}
\end{document}